\documentclass[runningheads,a4paper]{llncs}

\usepackage[utf8]{inputenc}
\usepackage{graphicx}
\usepackage{amsmath}
\usepackage{amssymb}
\usepackage{mathtools}
\usepackage{courier}
\usepackage{listings}
\usepackage{hyperref}
\usepackage{verbatim}
\usepackage{mdwlist}

\usepackage{url}
\urldef{\mailsa}\path|cooper@math.sc.edu|
\urldef{\mailsb}\path|overstrr@email.sc.edu|
\newcommand{\keywords}[1]{\par\addvspace\baselineskip
\noindent\keywordname\enspace\ignorespaces#1}

\newtheorem{prop}[theorem]{Proposition}

\newtheorem{cor}[theorem]{Corollary}

\newcommand{\cH}{\mathcal{H}}

\newcommand{\cF}{\mathcal{F}}

\newcommand{\cQ}{\mathcal{Q}}
\newcommand{\cS}{\mathcal{S}}
\newcommand{\cT}{\mathcal{T}}

\newcommand{\FF}{\mathbb{F}}

\newcommand{\NN}{\mathbb{N}}

\newcommand{\ZZ}{\mathbb{Z}}

\newcommand{\pyth}{\textsc{pyth}}
\newcommand{\prim}{\textsc{prim}}
\newcommand{\schur}{\textsc{schur}}

\title{Coloring so that no Pythagorean Triple is Monochromatic}

\titlerunning{Coloring Pythagorean Triples}

\author{Joshua Cooper \and Ralph Overstreet}

\institute{University of South Carolina,
Department of Mathematics,\\
1523 Greene Street,
Columbia, SC 29208 \\
\mailsa,
\mailsb \\
\url{http://www.math.sc.edu/}}

\begin{document}

\maketitle

\begin{abstract}
We address the question of the ``partition regularity'' of the Pythagorean equation \(a^2+b^2=c^2\); in particular, can the natural numbers be assigned a 2-coloring, so that no Pythagorean triple (i.e., a solution to the equation) is monochromatic?  We prove that the hypergraph of Pythagorean triples can contain no Steiner triple systems, a natural obstruction to 2-colorability.  Then, after transforming the question into one about 3-CNF satisfiability and applying some reductions, a SAT solver is used to find a 2-coloring for \(\{1,\ldots,7664\}\).  Work continues as we seek to improve the reductions and extend the computation.
\end{abstract}

\section{Introduction}
One of the central problems in Ramsey Theory on the integers (see, for example, \cite{LR04}) is the following question: given an equation -- or set of equations -- and a positive integer \(k\), does there exist a coloring of \(\NN\) so that there are no monochromatic solutions to the equation(s)?  That is, suppose the family \(\cF\) of subsets of \(\NN\) is defined to be the set of solutions \(\{x_1,\ldots,x_n\}\) to a system of equations in \(n\) variables.  Does there exist a function \(c : \NN \rightarrow [k]\) so that, for each \(F \in \cF\), there are \(x_i,x_j \in F\) with \(c(x_i) \neq c(x_j)\)?  A system of equations is called ``partition regular'' if, for any number of colors \(k\), every \(k\)-coloring of the integers admits a monochromatic solution.  Perhaps the most famous examples of partition regularity are the equations \(x + y = z\) (Schur's Theorem) and \(x + y = 2z\) (van der Waerden's Theorem); Rado's Theorem provides a vast generalization of these examples.

A question that is surprisingly resistant to extant methods is the partition regularity of the ``Pythagorean equation'' \(x^2 + y^2 = z^2\); see, e.g., \cite{CFHW14}.  To date, it is not even known if it is possible to 2-color the naturals so that there are no monochromatic Pythagorean triples.  However, there is a natural translation of this question into that of solving a certain 3-CNF satisfiability problem:
\begin{equation}
\bigwedge_{\substack{i,j,k \\ i^2 + j^2 = k^2}} \left ( x_i \vee x_j \vee x_k \right ) \wedge \left ( \neg x_i \vee \neg x_j \vee \neg x_k \right ) . 
\end{equation}
In this case, ``true'' and ``false'' assignments are the colors given to the indices of variables.

In order to show that the Pythagorean triples are partition regular, it suffices to demonstrate a subfamily of them which is not bipartite, i.e., any $2$-coloring induces a monochromatic triple.  (The existence of such a family follows from the de Bruijn-Erd\H{o}s Theorem, q.v. \cite{O83}.)  It is known that no two Pythagorean triples share two points, i.e., no leg and hypotenuse of one right triangle are the two legs of another integer right triangle.  Indeed, the two equations $x^2 + y^2 = z^2$ and $y^2 + z^2 = t^2$ yield two squares ($y^2$ and $z^2$) whose sum and difference are also squares.  One can find an elementary proof that no such pair of squares exists in Sierpi\'nski's classic text \cite{S88}.

Therefore, in our search for nonbipartite subsystems of triples, we need only consider those in which each pair of triples intersect in at most one point.  Such families are known as ``linear $3$-uniform hypergraphs'' (in the graph theory community) and ``partial triple systems'' or ``packings'' (in design theory).  The smallest nonbipartite partial triple system is the ubiquitous Fano plane $F_7$, the projective plane over $\FF_2$ with seven points and seven triples.  $F_7$ is also a ``Steiner triple system,'' meaning that each pair of points is contained in exactly one triple.  In fact, it is no hard to see that every Steiner triple system is not bipartite, so it is reasonable to search for these classic obstructions to 2-colorability among the Pythagorean triples.  However, we show below that a broad class of triple systems including the family $\pyth$ of all Pythagorean triples contains no Steiner triple systems.  Therefore, any search for nonbipartite partial triple systems in $\pyth$ necessarily must consider other systems than Steiner triple systems.

In the next section, we introduce some notation and definitions.  Section \ref{ref:noSTS} contains the proof that $\pyth$, and all ordered triple systems with the ``sum property'', contain no Steiner triple systems.  In the following sections, we describe the satisfiability solving methodology used to show that at least \([7664]\) is \(2\)-colorable without any monochromatic Pythagorean triples.  We conclude with some directions for further research.

\section{Preliminaries}

A {\it triple system} $\cH$ is a pair $(V,E) = (V(\cH),E(\cH))$ consisting of a vertex set $V$ and a family of unordered triples $E \subset \binom{V}{3}$.    A {\it partial triple system}, aka a {\it 3-uniform linear hypergraph}, is a triple system in which each pair of distinct edges intersect in at most one vertex.  A {\it Steiner triple system} is a partial triple system $\cS$ so that, for each pair of vertices $x$ and $y$, there is some $z$ so that $\{x,y,z\}$ is an edge of $\cS$.  An {\it ordered triple system} $\cH$ is a triple $(V,E,<) = (V(\cH),E(\cH),<_{\cH})$ consisting of a vertex set $V$, a family of unordered triples $E \subset \binom{V}{3}$, and a total ordering $<$ on $V$.  We say that an ordered triple system $\cH$ has the {\it sum property} if, whenever $\{a,b,c\}$ and $\{a^\prime,b^\prime,c^\prime\}$ are two edges with respective maxima $c$ and $c^\prime$,
\begin{equation}
(a \leq a^\prime) \wedge (b < b^\prime) \Rightarrow (c < c^\prime).
\end{equation}
\noindent {\bf Examples}:
\begin{enumerate}
\item Let $V = \ZZ$, and, for each $a \neq b$, let $e = (a,b,a+b)$ be an edge of $E$.  We call this the {\it Schur Triple System}, and denote it by $\schur$.  It is easy to see that $\schur$ has the sum property.  This shows immediately that the sum property does not imply finite colorability, since Schur's Theorem says that any coloring $\chi : \NN \rightarrow [c]$ with finitely many colors $c$ admits a monochromatic $e \in \schur$, i.e., $|\chi(e)| = 1$.  Such a map is called a (weak) hypergraph coloring.   A map $\chi : \NN \rightarrow [c]$ so that $|\chi(e)| = 3$ instead of just $|\chi(e)| > 1$ is known as a ``strong coloring,'' and corresponds exactly to a proper vertex coloring of the complement of the ``leave,'' i.e., the graph of all pairs contained in some triple $e \in E(\cH)$.  Some areas of mathematics call this the ``shadow'' graph or ``$1$-skeleton'' of $\cH$, and denote it by the boundary operator $\partial \cH$.
\item It is clear that any order-preserving isomorphic image of a triple system with the sum property also has the sum property, and that any subhypergraph of a triple system with the sum property does as well.  For example, we define the {\it Pythagorean Triple System} $\pyth$ by $V(\pyth) = \NN$ and $E(\pyth) = \{\{a,b,c\} : a^2 + b^2 = c^2\}$.  Since one can embed $\pyth$ into $\schur$ monotonically by the map $n \mapsto n^2$, $\pyth$ has the sum property as well.  It is a wide open problem to determine whether $\pyth$ has a weak coloring with finitely many colors.  (It is even open whether it is strongly colorable.)  As mentioned in the introduction, $\pyth$ is actually linear, i.e., no two edges intersect in more than one vertex.
\item A special subsystem of $\pyth$ is the {\it Primitive Pythagorean Triple System} $\prim$ consisting of all Pythagorean triples which are relatively prime.  That is, $V = \NN$ and
    \begin{equation}
    E(\prim) = \{\{a,b,c\} : a^2 + b^2 = c^2 \textrm{ and } \gcd(a,b,c) = 1\}.
    \end{equation}
    It is easy to see that $\pyth$ is actually a union of dilates of $\prim$ by each $d \in \NN$.  However, $\prim$ is bipartite: the parity coloring $n \mapsto n \pmod{2}$ provides a $2$-coloring.
\end{enumerate}

We define a special class of partial triple systems called ``bicycles.''  The $k$-bicycle has $2k+2$ vertices and $2k$ edges.  Its vertices are the elements of $\ZZ_{2k}$ and two ``antipodes'' $a$ and $b$; its edges are all triples of the form $\{a,2j,2j+1\}$ and $\{b,2j-1,2j\}$, $0 \leq j < k$.  The $2$-bicycle is also known as the {\it Pasch configuration}, or {\it quadrilateral}: the six-point partial triple system consisting of the edges $abc$, $ade$, $bef$, $cdf$.  The $3$-bicycle appears in the literature as the ``hexagon'' (e.g., \cite{CF08}): eight points $\{a,b,d,e,f,g,h,i\}$ with edges $\{afh,aei,adg,beh,bdi,bfg\}$.

\begin{figure}
\centering
\includegraphics[scale = .4]{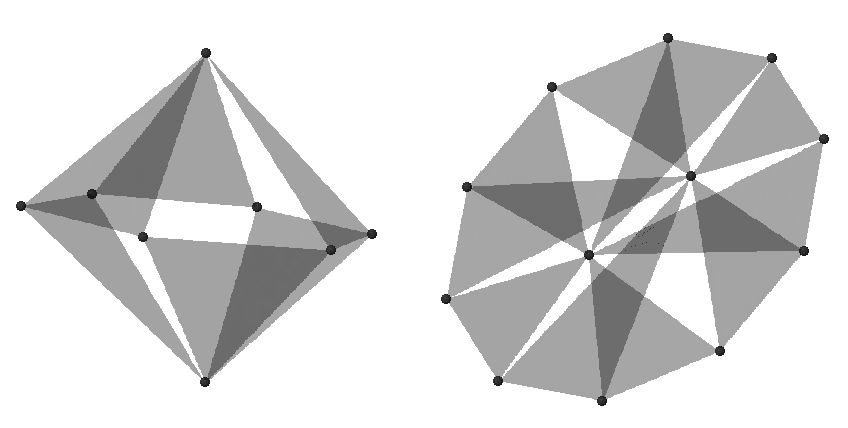}
\caption{A $3$-bicycle and a $5$-bicycle.} \label{fig:bicycles}
\end{figure}

The following proposition follows immediately from well-known results in the theory of triple systems.  (See, for example, \cite{CCR83}.)  For completeness, we give a short proof.

\begin{lemma} \label{prop:stshasbicycle} If $v,w$ are vertices of nontrivial Steiner triple system $\cS$, then there is a $k$-bicycle for some $k \geq 2$ in $\cS$ with antipodes $v$ and $w$.
\end{lemma}
\begin{proof} Define the ``link'' $\cS_x$ of a vertex $x \in \cS$ to be the set of pairs $\{a,b\}$ so that $\{a,b,x\}$ is an edge of $\cS$.  For some $z \in \cS$, $\{v,w,z\}$ is a triple of $\cS$, so $\cS_v$ consists of a perfect matching $M_1$ on $\cS - \{v,w,z\}$, plus the edge $\{w,z\}$; similarly, $\cS_w$ consists of a perfect matching $M_2$ on $\cS - \{v,w,z\}$, plus the edge $\{v,z\}$.  The union of $M_1$ and $M_2$ is composed of even-length cycles of length at least $4$; any one of these forms a bicycle with antipodes $v$ and $w$.
\end{proof}

We now define two weaker implicants of the sum property.  We say that an ordered triple system has the {\it upper sum property} if, whenever $\{a,b,c\}$ and $\{a,b^\prime,c^\prime\}$ are two edges with $c$ and $c^\prime$ their respective maxima,
\begin{equation}
(b > b^\prime) \Rightarrow (c > c^\prime).
\end{equation}
This clearly follows from the sum property by setting $a = a^\prime$.  We say that an ordered triple system has the {\it lower sum property} if, whenever $\{a,b,c\}$ and $\{a^\prime,b^\prime,c\}$ are two edges with $c$ as both of their maxima,
\begin{equation}
(a > a^\prime) \Rightarrow (b < b^\prime).
\end{equation}
To see that the lower sum property follows from the sum property, suppose that $a > a^\prime$ but $b > b^\prime$.  Then it follows that the maximal element of $\{a,b,c\}$ is greater than the maximal element of $\{a^\prime,b^\prime,c\}$, whence $c > c$, a contradiction.  We say that two pairs of integers $\{a,b\}$ and $\{c,d\}$ are ``nesting'' if $a < c < d < b$.  Then it is possible to restate the upper sum property as the fact that, for each vertex $x$, the subset of the link graph $\cH_x$ intersecting $\{y : y \geq x\}$ is a non-nesting matching.  The lower sum property may be similarly restated as the fact that, for each vertex $x$, the subset of the link graph $\cH_x$ induced by $\{y : y \leq x\}$ is a fully nested matching, i.e., for each two edges $e$ and $f$, $e$ is nested in $f$ or vice versa.

\section{Sum Property implies No STS} \label{ref:noSTS}

\begin{prop} \label{prop:noquadupper} If $\cH$ has the upper sum property, and $\cQ$ is a $k$-bicycle in $\cH$, then the maximal two points of $\cQ$ are not its antipodes.
\end{prop}
\begin{proof} Let
\begin{equation}
\cQ = (\ZZ_{2k} \cup \{a,b\},\{a01,a23,a45,\ldots\}\cup\{b12,b34,b56,\ldots\}) \subset \cH
\end{equation}
be a $k$-bicycle, and suppose $a$ and $b$ are the maximal two points of $\cQ$.  We may assume without loss of generality that $a > b$.  Then the maximal elements of all triples are $a$ or $b$ (depending on which antipode they contain).  Therefore, since $\{a,2j,2j+1\} \cap \{b,2j+1,2j+2\} = \{2j+1\}$,
\begin{equation}
(a > b) \Rightarrow (2j > 2j+2)
\end{equation}
for each $0 \leq j < k$.  However, the quantities above are modulo $2k$, whence the set of resulting inequalities is circular and therefore inconsistent.
\end{proof}

\begin{prop} \label{prop:noquadlower} If $\cH$ has the lower sum property, and $\cQ$ is a $k$-bicycle in $\cH$, then the maximal two points of $\cQ$ are not its antipodes.
\end{prop}
\begin{proof} Let
    \begin{equation}
\cQ = (\ZZ_{2k} \cup \{a,b\},\{a01,a23,a45,\ldots\}\cup\{b12,b34,b56,\ldots\}) \subset \cH
    \end{equation}
be a $k$-bicycle, and suppose $a$ and $b$ are the maximal two points of $\cQ$.  We may assume without loss of generality that $a > b$.  Then the maximal elements of all triples are $a$ and $b$ (depending on which antipode they contain).  Since $\{a,2j,2j+1\}$ is an edge for each $0 \leq j < k$, the $k$ pairs $\{2j,2j+1\}$ are a matching, and they are linear ordered by nesting, by the lower sum property.  Suppose, without loss of generality, that $\{0,1\}$ is the outermost matching edge and $0 < 1$, so that, for all $x \in \ZZ_{2k} \setminus \{0,1\}$, $0 < x < 1$.  The pairs $\{-1,0\}$ and $\{1,2\}$ are also nested, since they arise from the edges $\{b,-1,0\}$ and $\{b,1,2\}$.  Hence, $-1 < \{1,2\} < 0$ or $1 < \{-1,0\} < 2$, and each of these possibilities contradicts $0 < 1$.
\end{proof}

\begin{cor} If $\cH$ has the full, lower, or upper sum property, then it does not contain any Steiner triple system.  In particular, $\pyth$ contains no Steiner triple system.
\end{cor}
\begin{proof} Suppose $\cH$ contained some Steiner triple system $\cT$.  Let $a$ and $b$ be the maximal elements of $\cT$.  Then $a$ and $b$ are the antipodes of some bicycle, by Lemma \ref{prop:stshasbicycle}.  However, this contradicts Propositions \ref{prop:noquadupper} and/or \ref{prop:noquadlower}.
\end{proof}

Note that a triple system with the sum property {\em can} contain a quadrilateral: for example, $\schur$ contains $\{5,15,20\}$, $\{5,8,13\}$, $\{7,8,15\}$, $\{7,13,20\}$.

\section{The Computational Approach}

We now describe the methodology used to certify the 2-colorability of the induced hypergraph of \(\pyth\) on the vertex set \([7664]\).

  \subsection{Listing the Triples up to N}
    To make a list of all Pythagorean triples having only elements less than or equal to some upper bound N, we use a function based on Dickson's Method, from Wikipedia \cite{Dickson}. Our function for generating triples uses the equations $r^2 = 2st$ for $r, s$, and $t \in  \mathbb{N}$, and $x = r+s, y = r+t$ and $z = r+s+t$, with $x^2+y^2=z^2$. With these equations, we step through all combinations of $s$ and $t$ so that $r \in \mathbb{N}$ and $z\leq$ N. The Python code is short.
    \begin{lstlisting}
    import numpy as np
    def dicksonGenTriples(uBound):
        data = []
        sSteps = range(1,uBound)
        for s in sSteps:
            tSteps = range(s,uBound-s)
            for t in tSteps:
                r = np.sqrt(2*s*t)
                if r == np.round(r):
                    r = int(r)
                    z = r + s + t
                    if z >= uBound:
                        break
                    x = r + s
                    y = r + t
                    if [x,y,z] not in data:
                        data.append([x,y,z])
        with open(`triples-dickson.json',`w') as f:
            json.dump(data, f)
        return data
    \end{lstlisting}
  \subsection{Writing the CNF}
  We remapped the integers involved in the Pythagorean triples to consecutive integers starting with 1 since, some SAT solvers require it. The bijection for remapping is represented in the Python code by a dictionary.
    \begin{verbatim}
  {"3": 1, "4": 2, "5": 3, "6": 4, "8": 5, "10": 6}
    \end{verbatim}
    The following is a CNF file for the triples with integers up to 10. These are [3,4,5] and [6,8,10].
    \begin{lstlisting}
        c 10
        p cnf 6 4
        1 2 3 0
        -1 -2 -3 0
        4 5 6 0
        -4 -5 -6 0
    \end{lstlisting}
    A full specification for CNF (Conjunctive Normal Form) is given at \cite{CNF}.
  \subsection{Choosing a Solver}
    We experimented with glueSplit\_clasp, Dimetheus, SWDiA5BY, PeneLoPe, Lingeling, and Treengeling, all from \url{http://www.satcompetition.org/} in 2014. Treengeling is a parallel solver that runs on some or all of the cores of a single machine, and seemed to be the fastest for a CNF corresponding to the set of triples up to about 7000. These runs completed within a few hours on each of the many solvers. We did not emphasize finding the best solver.
    \subsection{Reducing the Size of the Input to the Solver}
    A natural goal of SAT solving is to ease the burden on the solver by finding parts of the hypergraph, corresponding to the solver's original CNF, which can be safely removed without changing the satisfiability of the new, reduced CNF. We have had to rely on these reductions for upper bounds above 7620. 

    Our only success to date at reducing the size of the input to the solver in this way, is ``pendant removal''. By pendant, we mean any edge (triple) of the 3-uniform hypergraph which has at least one vertex not contained in any other edges. Once a triple system with a pendant removed is properly 2-colored, the pendant's vertex, or vertices, of degree one can then be assigned any color that is not assigned to some other vertex of its edge (of which there is always at least one). Pendants are always removable and can be removed iteratively until no more pendants remain. Once these pendants have been removed from a graph and the remaining graph has been fed to the solver and colored, it is easy to add back the removed pendant edges and 2-color them as they are added back. The edges are added back in the reverse order that they were removed.
\section{Parallelizing Across Multiple Computer Cluster Nodes}
  \subsection{The Overall Strategy}
    We use a strategy to parallelize across multiple nodes of a cluster similar to one used in a paper by Biere, Heule, Kullmann, and Wieringa \cite{BHKW12}. In particular, we convert the problem into $2^m$ many potentially easier problems, by making $2^m$ copies of the original CNF file and appending a distinct list of extra clauses to each one. We run these new CNF files on up to $2^m$ cluster nodes; one node per CNF. As soon as any node completes, if that CNF is satisfiable, we have our solution and can force all the other nodes to terminate before completing. The various nodes do not communicate amongst themselves and run independently from one another.
    
    To create the extra clauses for each of the $2^m$ copies of the original CNF file, we select $m$ special vertices (integers), each coming from some Pythagorean triple which was used to generate the original CNF. 
    For each special vertex, $i$, we make a new clause
    \begin{equation}
    ( \neg x_i \vee \neg x_i \vee \neg x_i )
    \end{equation}
    or
    \begin{equation}
    ( x_i \vee x_i \vee x_i ).
    \end{equation}
    This gives us $2^m$ different lists of $m$ extra clauses to add to each of the $2^m$ copies of the original CNF file. These extra $m$ clauses provide the solver with predetermined truth assignments for specific vertices.  Either one of these $2^m$ runs of the SAT solver, on the slightly varying CNF files, will complete and provide a valid coloring, or all the runs will complete and each certify unsatisfiability -- given sufficient computational time.
    \subsection{An Example CNF File Modified for Parallelization}
    We present one of the 4 CNF files for the list of triples [[3,4,5],[6,8,10]], split on 2 vertices; 3 and 4, which are remapped to 1 and 2. 
    The line which begins with p is also updated by adding $m = 2$ to the last integer in the original CNF file on that line. The original CNF had a 4 in that spot. Note especially the last 2 clauses.
    \begin{lstlisting}
    c 10
    p cnf 6 6
    1 2 3 0
    -1 -2 -3 0
    4 5 6 0
    -4 -5 -6 0
    1 1 1 0
    -2 -2 -2 0
    \end{lstlisting}
  \subsection{Choosing Vertices to Split On} 
    It would be desirable to choose a list of $m$ special vertices so that each vertex in the list is relatively independent from the others. By independence, we mean that any assignment of truth values to the variables corresponding to special vertices leads to a similar number of (partial) satisfying assignments.  We have attempted to gauge independence by measuring the time needed for each of the $2^m$ truth value combinations to run using Treengeling. The minimum distance in the hypergraph of all Pythagorean triples, up to an upper bound, $N$, between any two of the special vertices is a reasonable proxy for their independence. Under current technical limitations, we have found $m = 2$ to run the fastest. 
    \subsubsection{The BFS Method of Finding Distant Vertices}
    A promising method of choosing some $v_1$ and $v_2$ which are distant from one another in the hypergraph of Pythagorean triples is to use what we call the breadth-first search (BFS) method. By way of illustration, we present Figure \ref{fig:nbrsgraph} and Table \ref{tab:nbrs}.  We start from a seed triple $t_0$, and use the method to find a $t_{11}$ that is furthest from $t_0$.
    \begin{figure}
      \begin{center}
	\includegraphics[width=12.0cm,height=6.0cm]{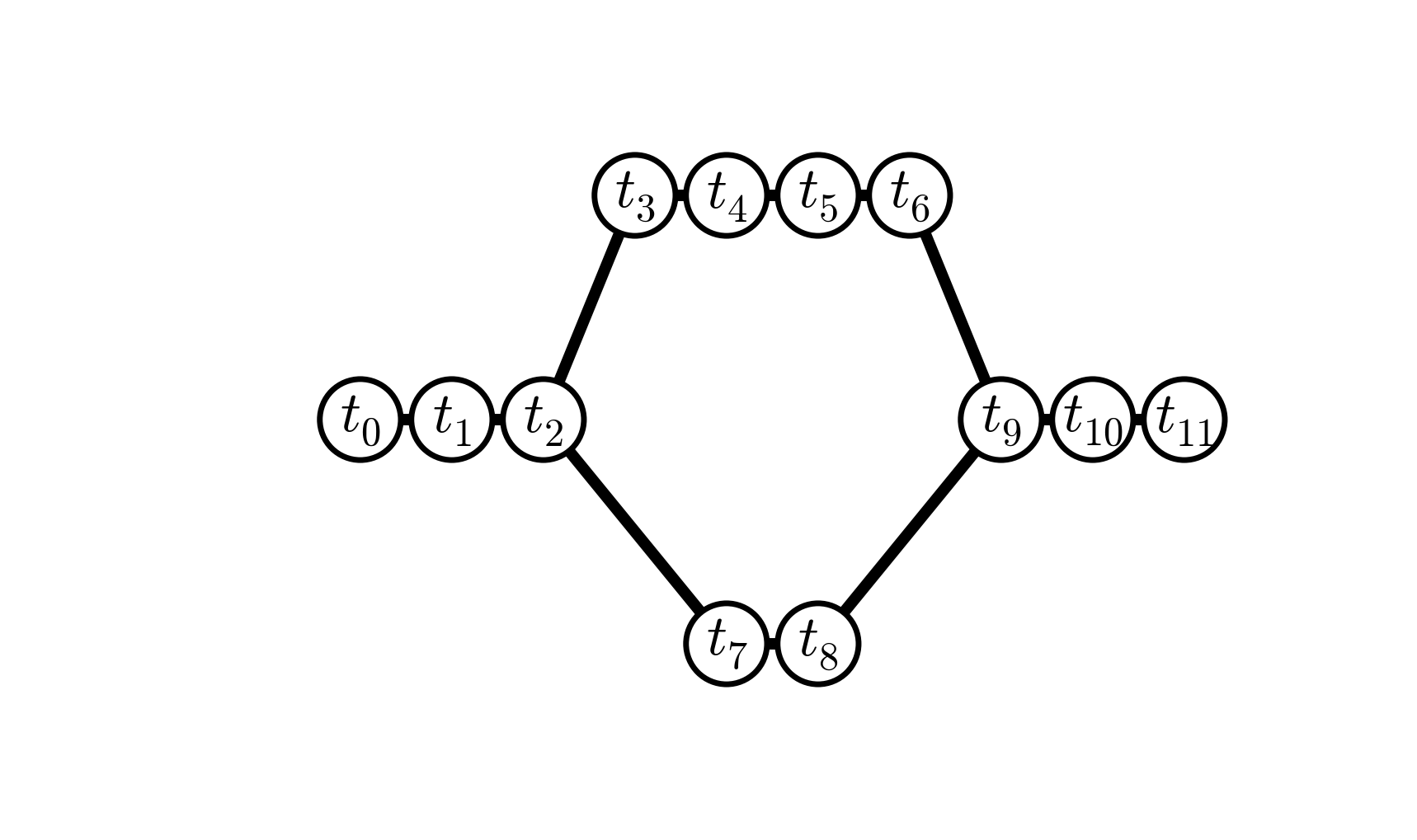}
	\caption{A graph of some hypothetical list of triples where any 2 triples share an edge if they intersect. For example $t_0$ = [3,4,5] and $t_1$ = [5,12,13] share an edge in a graph of this type, as the 2 triples share a vertex.} 
	\label{fig:nbrsgraph}
      \end{center}
     \end{figure}
     \begin{table}
    \centering
    \caption{The triples of each BFS level corresponding to the graph in Figure \ref{fig:nbrsgraph}.}
    \begin{tabular}{ c c c c c c c c c}
  level & 0 & 1 & 2 & 3 & 4 & 5 & 6 & 7 \\
   & $t_0$ & $t_1$ & $t_2$ & $t_3$ & $t_4$ & $t_5$ & $t_6$ & $t_{11}$ \\
    & & & & $t_7$ & $t_8$ & $t_9$ & $t_{10}$ & \\
   \end{tabular}
    \label{tab:nbrs}
    \end{table}
     Once a triple appears in a level, it will not appear again in the list. When there is more than one path from the seed to some other triple, the listing in Table \ref{tab:nbrs} shows only the shortest path. Thus any two triples which are in two distant levels from one another, necessarily have a high minimum graph distance between them.
     
    In the Pythagorean Triple System, to find two vertices that are relatively far apart, we choose a seed triple with minimum constituent integers as small as possible. In the example of Figure \ref{fig:nbrBars}, we run the BFS method on a list of triples that has had its pendants removed.
    \begin{figure}
      \begin{center}
	    \includegraphics[width=12.0cm]{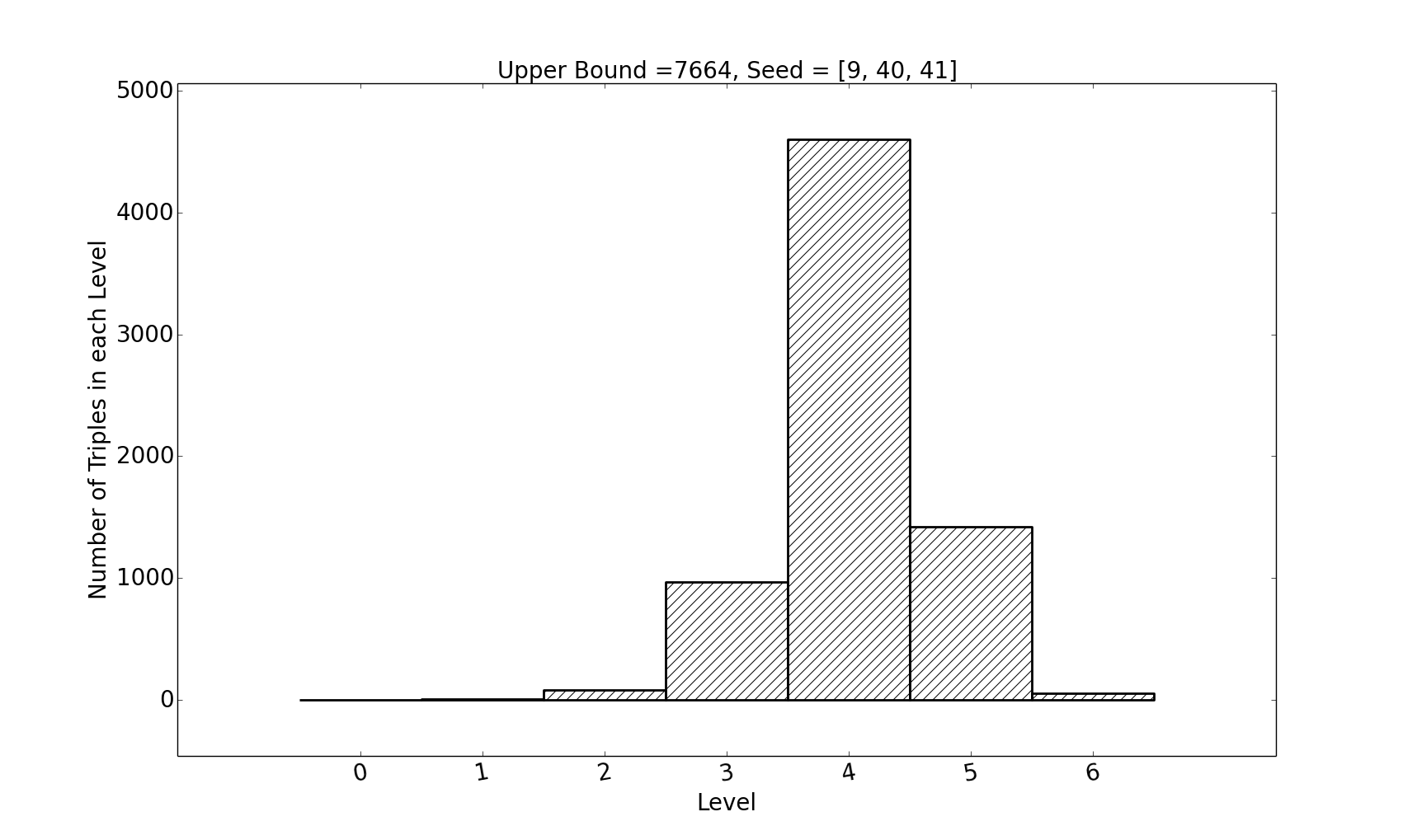}
	    \caption{The bar plot shows the size of each BFS level for the list of pendant-removed triples with all vertices less than or equal to 7664.}
	    \label{fig:nbrBars}
      \end{center}
    \end{figure}
    The plot gives an indication of the proximity of two vertices to each other and for the connectedness, loosely defined, of each of the two vertices to the rest of the hypergraph.
    \subsubsection{Gauging the Independence of a List of Special Vertices}
    As an illustration of a methodology to gauge the independence of a set of vertices, we choose a list of 4 randomly selected special vertices and uniformly at random remove \( 10 \% \) of the triples in the list of triples with vertices less than or equal to some upper bound around say, 7621. This speeds up the SAT solver computations from about a day to about 100 seconds. Each run uses 1 node with 12 cores. Next we record the run time for 10 runs each, of the 16 truth value combinations, for the 4 special vertices.  Each of the 10 runs of the SAT solver uses a different reduced list of triples. There are a total of 10 reduced triple lists for the whole process. See Figure \ref{fig:times}.
    \begin{figure}
    \centering
    \includegraphics[width=12.0cm]{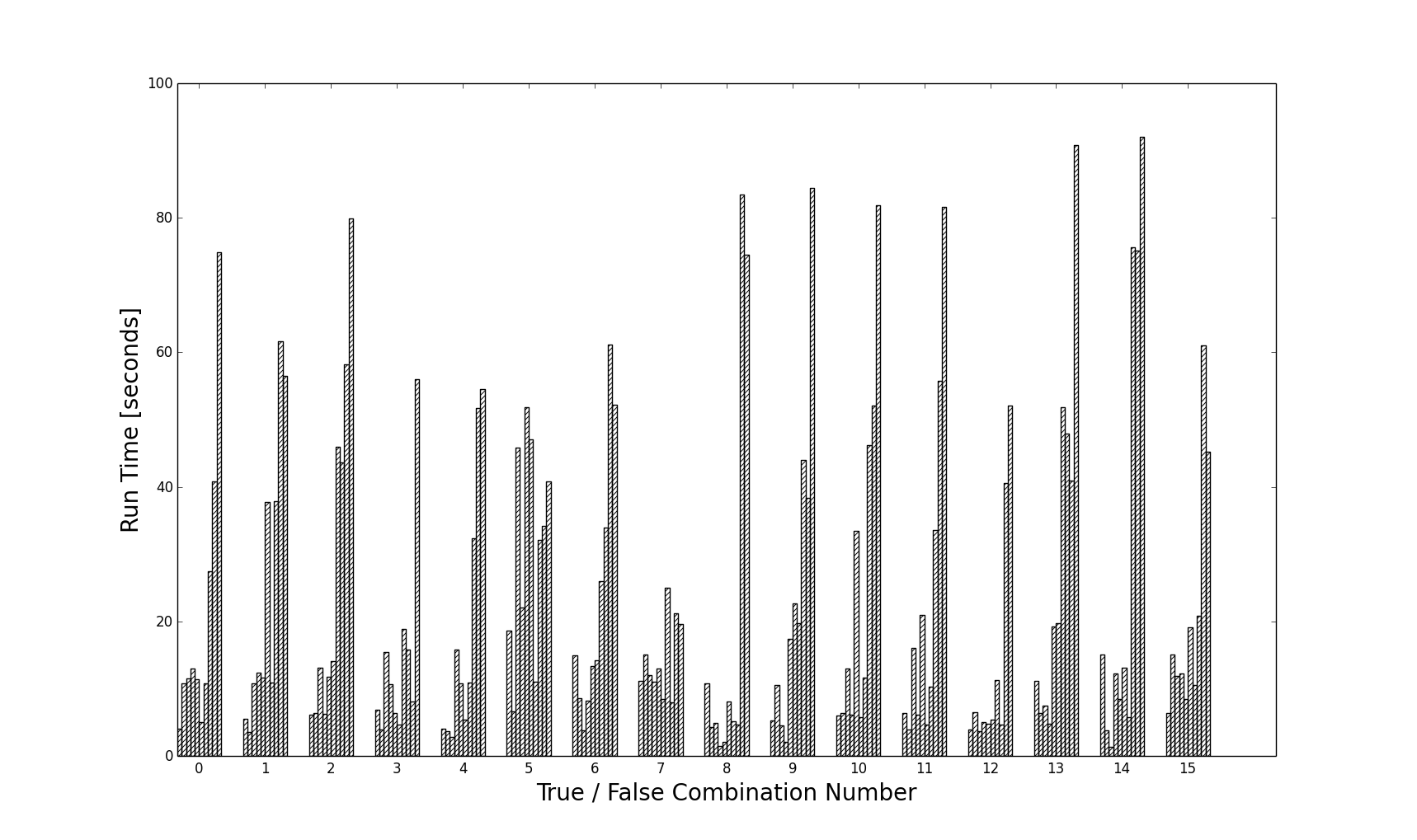}
    \caption{Run times of 10 runs of each of the 16 possible truth values of the 4 special vertices. The special vertices were chosen uniformly at random.}
    \label{fig:times}
    \end{figure}
    For each of the 16 truth value assignments, we compute an empirical mean run time; then, the sample variance of these 16 numbers is then computed.  For a set of chosen special vertices, a low variance indicates high independence, since the solver worked approximately equally hard across the different possible truth values.
  \subsection{Running a Node Pool on Clusters}
    We have devised a method for simultaneously running a fixed number of nodes, each running an instance of the solver on
    12 cores. 
    We are allowed to use up to about 10 of these nodes at a time, after waiting in the queue for some amount of time. See a detailed description of our clusters at the University of South Carolina, Maxwell and Planck, at\\\\
    \url{http://www.sc.edu/about/offices_and_divisions/division_of_information_technology/rci/hpc_resources/index.php}\\
    
    We use the two clusters interchangeably. Their specifications are as follows.
\begin{itemize*}
    \item Maxwell
    \begin{itemize*}
        \item Hardware 
        \begin{itemize*}
            \item 40 GL390 Nodes: 12 cores per node, Intel Xeon 2.4 GHz,                 24 GB RAM 
            \item 6 SL250: 16 cores per node, Intel Xeon 2.60GHz, 32 GB                  RAM 
            \item 1 DL380 Headnode: 12 core, 48GB RAM 
            \item 24TB attached storage
        \end{itemize*}
        \item Interconnect
        \begin{itemize*}
            \item QDR Infiniband
        \end{itemize*}
        \item Software
            \begin{itemize*}
                \item CentOS 
                \item HP CMU cluster management utility 
                \item OpenMPI 
                \item Torque/Maui scheduler
            \end{itemize*}
    \end{itemize*}
    \item Planck
    \begin{itemize*}
        \item Hardware
        \begin{itemize*}
            \item 20 SL250 Nodes: 12 cores per node, Intel Xeon 2.8                      GHz, 24 GB RAM
            \item 15 x 3 NVIDIA M1060 with 240 cores each, 3 x 3                     \item NVIDIA M2070 with 448 cores each
            \item 11 TB attached storage
        \end{itemize*}
        \item Interconnect
        \begin{itemize*}
            \item QDR Infiniband
        \end{itemize*}
        \item Software
        \begin{itemize*}
            \item CentOS
            \item HP CMU cluster management utility
            \item OpenMPI
            \item Torque/Maui scheduler
        \end{itemize*}
    \end{itemize*}    
\end{itemize*}
    Our parallelization starts from an outer shell script which runs a sequential Python program
    that creates a fixed number of inner shell scripts, each running 1 instance of Treengeling on one full node. The Python program, which is started by the outer shell script, uses
    system calls to detect the completions of each of the inner shell scripts by detecting the existence of the output file
    for each inner shell script. When an inner shell script completes, a new inner shell script is fed to the queue until all
    the desired instances of Treengeling have been submitted, or until one instance of Treengeling returns a valid coloring. 
  \subsection{Choosing the Number of Special Vertices}
    We present some evidence that splitting approximately 4 ways using 2 special vertices is fastest, given current technical constraints. See Figure \ref{fig:scalable}, which plots run time against the number of nodes at which the problem was split.
    \begin{figure}
      \begin{center}
	\includegraphics[width=12.0cm]{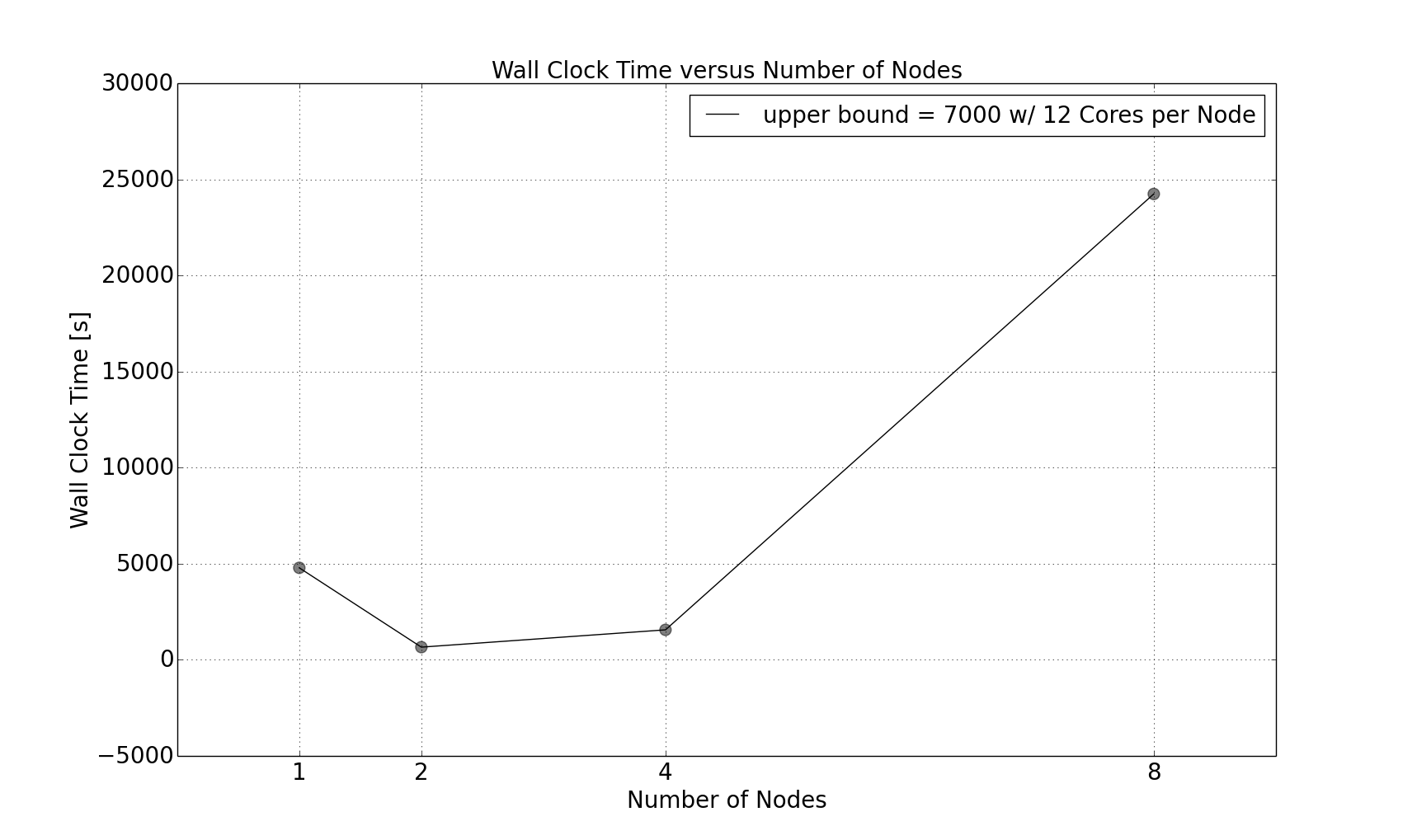}
	\caption{Some results of scalability testing.}
	\label{fig:scalable}
      \end{center}
    \end{figure}

\section{Results}
  \subsection{Performance of Treengeling}
    Treengeling can find solutions for all the Pythagorean triples with integers up to 7620 in about 21 hours. If
    the list of triples has been pendant removed, Treengeling can find a solution for all triples up to 7650 in about 33 hours. 
  \subsection{Freedom of Coloring and Searching for Patterns}
    We have looked for some patterns in the large valid colorings found, but have not found much. There seems to be a lot of freedom in how a coloring can be achieved, which seems to be a reason why we have not found any consistent patterns in the colorings. 
  
    \subsection{A Visualization of the Coloring up to 7664}
    The coloring in Figure \ref{fig:bigColoring} was made by removing the pendants from the list of triples with integers up to 7664, running the solver on the result, and adding back the triples, one at a time, in the reverse order that they were removed and 2-coloring them as they were added back. The final coloring was checked to see that every triple with all integers less than or equal to 7664 is bichromatic.
    
Integers which do not occur in any triple, with all constituent integers of the triple being less than or equal to  7664, are white. The solver colors every integer which it receives from any triple as either false or true. While some integers might be colorable as either false or true in a particular coloring, the solver will assign some color to each of these integers along with every other integer that it receives.
    \begin{figure}
      \begin{center}
	\includegraphics[width=12.0cm]{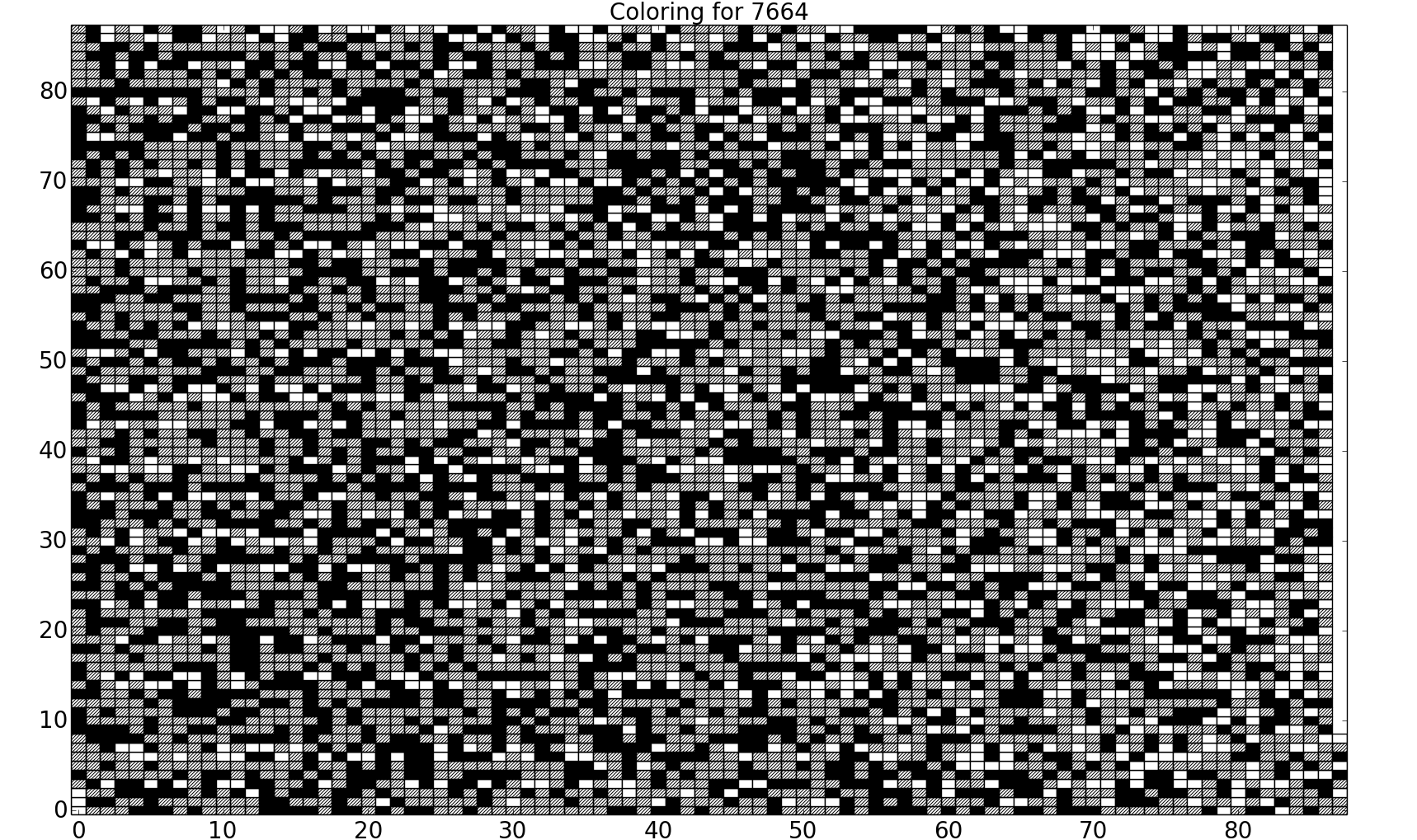}
	\caption{A coloring of the hypergraph of $\pyth$ induced by the vertex set [7664]. False is {\bf grey}. True is {\bf black}. Integers not appearing in the input to the solver are {\bf white}. The integer positions climb vertically by columns from the lower left to  the upper right.}
	\label{fig:bigColoring}
      \end{center}
    \end{figure} 

\section{Future Work}
  \subsection{Further Reducing the Input}
    We are running a sequential Python function on the clusters that systematically checks all possible linear 3-uniform
    hypergraphs on $n$ vertices to see if any further reductions to $\pyth$ can be made.
We have
    not yet found any ``removable'' subgraphs that are not pendants, from the triple list up to 7700.  We would like to parallelize this  sequential operation using Python and MPI, or C and MPI if need be.
    
\subsubsection{Acknowledgements}
Thanks to Ron Graham, Bill Kay, and Christopher Poirel for helpful discussions and ideas in the development of the present work. 
Paul Sagona has been helpful with the high performance computing resources at the University of South Carolina.

\end{document}